\documentclass{amsart}
\usepackage{latexsym,amsmath,amscd, amssymb, diagrams, color}
\allowdisplaybreaks
\usepackage{calrsfs}
\DeclareMathAlphabet{\pazocal}{OMS}{zplm}{m}{n}

\newtheorem{theorem}{Theorem}[subsection]
\newtheorem{corollary}[theorem]{Corollary}
\newtheorem{lemma}[theorem]{Lemma}
\newtheorem{proposition}[theorem]{Proposition}
\theoremstyle{definition}
\newtheorem{definition}[theorem]{Definition}
\theoremstyle{emp}

\theoremstyle{question}

\theoremstyle{conj}

\theoremstyle{notat}

\theoremstyle{claim}

\theoremstyle{construction}

\theoremstyle{remark}
\newtheorem{remark}[theorem]{Remark}
\theoremstyle{example}

\numberwithin{equation}{subsection}
\title{A note on the algebra of p-adic multi-zeta values }
\author{S\.{i}nan \"{U}nver}
\address{Ko\c{c} University, Mathematics Department. Rumelifeneri Yolu, 34450, Istanbul, Turkey}
\email{sunver@ku.edu.tr}

\begin{document}
\maketitle
\noindent

\begin{abstract} 
We prove that the algebra of p-adic multi-zeta values, as defined in \cite{U1} or \cite{fu}, are contained in another algebra which is defined explicitly in terms of  series. The main idea  is to truncate certain series, expand them in terms of series all of which are divergent except one, and then take the limit of the convergent one. The main result is Theorem 3.0.27. 
\end{abstract}

\section{Introduction} 

Multi-zeta values were defined by Euler as the sum of the series: 
$$
\zeta(s_{1},s_{2}, \cdots,s_{k}):=\sum _{0<n_{1}<\cdots <n_{k}}\frac{1}{n_{1} ^{s_{1}} n_{2} ^{s_{2}} \cdots n_{k} ^{s_{k}}},
$$
for $s_{1}, \cdots, s_{k-1} \geq 1$ and $s_{k}>1.$  The Euler-Kontsevich formula \cite{go} expresses  these numbers in terms of iterated integrals on the thrice punctured line $X:=\mathbb{P}^{1} \setminus \{0, 1, \infty \}  .$ This formula interprets multi-zeta values as real periods of the mixed Tate motive coming from the unipotent fundamental group of $X$ \cite{de}, \cite{go}. 
These imply many relations between the multi-zeta values and the algebra of these values has arithmetic significance as it relates to the tannakian fundamental group  of mixed Tate motives over $\mathbb{Z}$  \cite{go}.

The p-adic version of the these values were defined by Deligne (unpublished, explained in \cite{U1}) coming from the comparison theorem between the de Rham  and the crystalline fundamental group of $X.$ The double shuffle relations for these values were proved by Furusho and Jafari and the Drinfel'd-Ihara relations were proved by the author. The question of algebraic independence or even non-vanishing seems to be a more difficult question. We do not even know that $\zeta_{p}(2k+1)$ are non-zero for all primes $p,$ and positive integers $k.$ This suggests that in order to prove linear independence among these numbers one might need a somewhat more explicit description of these numbers or the algebra generated by these numbers. This paper is a first attempt in this direction. Our aim in this note is to prove that the algebra of p-adic multi-zeta values lie in another algebra that is described only using certain series. The remaining, but more difficult question is to study these series.

It turns out that the individual truncated series that appear in the expression for p-adic multi-zeta values are divergent, but their certain linear combinations converge. The main  idea is to consider these divergent series and express them as  linear combinations of a convergent series and  other divergent series and then take the convergent part. We call this process the {\it regularization} of the series. 

In \textsection{2}, we describe this process. The main point is that there are certain simple divergent series which we denote by $\sigma_{p}(\underline{s}),$ which are linearly independent  under the algebra of {\it power series functions}. That these series are linearly independent is proved in Proposition 2.0.3; the fact that all the series we are interested can be expressed in terms of these divergent series is proved in Proposition 2.0.7.  Next in Proposition 2.0.14, which forms the basis for the inductive arguments, we prove that, under  conditions that are satisfied by the series that appear below,  if the truncated series that are the coefficients of a power series are regular  then so are the coefficients of its antiderivative after it is multiplied by one of the forms $\omega_{i}.$ 

In \textsection 3, we apply the above results to the truncated series that appear in the expressions for p-adic multi-zeta values. The main idea is based on the results of \cite{U1}. Namely, using the standard lifting of frobenius on $X,$ which is a good choice outside a disc of radius 1 around 1, we obtain a differential equation (3.0.3) for $\mathfrak{g}.$ Next we use \cite[Proposition 2]{U1} to find an expression for the value of $\mathfrak{g}$ at infinity and the relation (3.0.2) to relate that value to $g,$ whose coefficients are the p-adic multi-zeta values.  The main  result is Theorem 3.0.27 where we prove that the algebra of p-adic multi-zeta values $\pazocal{Z}$ lie inside the  algebra $\pazocal{P}$ of values of regularized series. Finally, we note in Remark 3.0.28 that the above result also implies the same statement for the p-adic multiple zeta values as defined by Furusho \cite{fu}, since the algebra that they generate is the same as $\pazocal{Z}.$
 
{\it Acknowledgements.} The author  thanks T. Terasoma for a very short discussion which became the main inspiration for this note. 

\section{Iterated sums}
Fix a prime $p.$   For $\underline{s}:=(s_{1},\cdots,s_{k}),$ with $0 \leq s_{i},$ and $\underline{m}:=(m_{1},\cdots,m_{k}),$ with $0 \leq m_{i}<p,$ let  
$$
\sigma(\underline{s};\underline{m})(n):=\sum \frac{1}{n_{1} ^{s_{1}} \cdots n_{k} ^{s_{k}}},
$$ where the sum is over $0<n_{1}<n_{2}<\cdots < n_{k}<n$ with $p|(n_{i}-m_{i}).$ Similarly, we let $\gamma(\underline{s};\underline{m})(n):=n^{-s_{k}}\cdot\sigma(\underline{s}';\underline{m}')(n),$ if $p|(n-m_{k})$ and 0 otherwise, where $\underline{s}'=(s_{1},\cdots,s_{k-1})$ and $\underline{m}'=(m_{1},\cdots,m_{k-1}).$    Let $\sigma_{p}(\underline{s})(n):=\sigma(\underline{s};\underline{0})(n),$ where $\underline{0}=(0,\cdots,0).$ We define the {\it depth}  as  $d(\underline{s})=k$ and the {\it weight} as $w(\underline{s}):=\sum s_{i}.$

 Let us call a function $f: \mathbb{\mathbb{N}}_{\geq n} \to \mathbb{Q}_{p},$ for some $n,$ a 
{\it power series  function}, if there exist power series $p_{i}(x) \in \mathbb{Q}_{p}[[ x]],$ which converge on the closed unit disc $D(r_{i})$ around 0, for some $r_{i} > |p|,$ for $0 \leq i < p ,$ such that $f(a) = p_{i} (a-i)$ for all $a \geq n$ and $p |(a-i).$ Clearly there is a unique power series function $\overline{f}$ with domain $\mathbb{Z}_{\geq 0}$ such that $\overline{f}$ restricts to $f$ on $\mathbb{N}_{\geq n}.$ We let $f(0):=\overline{f}(0),$ or more explicitly $f(0)=\lim_{k \to 0 \atop {k \geq n}}f(k).$ We will identify two power series functions if they are the same on the intersection of their domains of definition.     Note that we have the following \cite[Proposition 5.0.5]{U2}:

\begin{proposition}\label{basic}
(i) The product and sum of power series functions are also power series functions. So is the function defined as $f(n)=n^{s},$ if $p \nmid n;$ and $f(n)=0,$ if $p|n,$ for   $s \in \mathbb{Z}.$  

(ii) If $f$ is a power series function, let us define $f^{[1]}$ and $f^{(1)}$   as $f^{[1]}(n)=f^{(1)}(n)= (f(n)-f(0))/n,$ if $p|n;$  $f^{[1]}(n)=0,$  $f^{(1)}(n)=f(n)/n,$   if $p \nmid n.$ Then both  $f^{[1]}$ and $f^{(1)}$ are  power series functions. 

(iii)
If $f: \mathbb{N}_{\geq n_{0}} \to \mathbb{Q}_{p}$ is a power series function and if  we define 
$$
F(n):=\sum_{n_{0} \leq k \leq n} f(k),
$$
then so is $F.$ 
\end{proposition}
  
The following lemma on power series will be essential while we are proving the linear independence of the $\sigma_{p}$'s.
 
\begin{lemma}
 Let $f,g \in \mathbb{Q}_{p}[[z]]$ be two power series which are  convergent on $D(a),$ for some $a>1.$ Suppose that $g \neq 0,$ and let $h:=f/g.$   If there exist  $C_{i} \in \mathbb{Q}_{p}$ and $n\geq 1$ such that $h(z+1)-h(z)=\frac{C_{n}}{z^{n}}+\cdots +\frac{C_{1}}{z},$ for infinitely many  $z \in D(a)$      then $h$ is constant and  $C_i=0,$ for all $i.$  
 \end{lemma}
 
\begin{proof}
By the Weierstrass preparation theorem, if the above equality holds for infinitely many $z \in D(a)$ then it holds for all $z \in D(a),$ except for the zeros of the denominators involved.   
Suppose that $f \neq 0,$ then again by the Weierstrass preparation theorem  %(P-adic analysis a short course on recent work by N. Koblitz p. 21) 
there are polynomials $p,q \in \mathbb{Q}_{p}[z],$ and  power series $u, v \in \mathbb{Q}_{p}[[ z]] ,$ which converge and are nonzero on $D(a)$ such that $f=p\cdot u$ and $g=q\cdot v.$ Therefore the poles  of $h(z)$ and $h(z+1)$  in $D(a),$ together with their multiplicities, are exactly those of $p(z)/q(z)$ and $p(z+1)/q(z+1).$  Let $\mathfrak{P}(k)$ denote the set of poles of $k$ in $D(a).$ 

Then 
$$
\mathfrak{P}(h) \triangle (\mathfrak{P} (h)-1) \subseteq \mathfrak{P} (\frac{C_{n}}{z^{n}}+\cdots + \frac{C_{1}}{z}) \subseteq \mathfrak{P} (h) \cup (\mathfrak{P} (h)-1),
$$
where $\Delta$ denotes the symmetric difference. 

If $C_{i} \neq 0$ for some $i,$ then $\mathfrak{P} (\frac{C_n}{z^{n}}+ \cdots +\frac{C_{1}}{z})=\{0\},$ hence $\mathfrak{P}(h)\neq \emptyset.$ This implies that $\mathfrak{P}(h) \neq \mathfrak{P}(h)-1.$ Since $|\mathfrak{P}(h)|=|\mathfrak{P}(h)-1|$ is finite,  this implies that the symmetric difference of $\mathfrak{P}(h)$ and $\mathfrak{P}(h)-1$ contains at least two elements. This is a contradiction. Hence $C_i=0,$ for all $i.$ This implies that $h(z+1)=h(z).$ Choosing an $\alpha$ where $h$ does not have a pole and replacing $h$ with $\overline{h}:=h-h(\alpha),$ we see that $\overline{h}$ has infinitely many zeros in $D(a)$ and hence is 0.  
 \end{proof}

 Let $\mathcal{P}$ denote the algebra of power series  in $\mathbb{Q}_{p}$ which converge on $D(r)$ for some $r>|p|.$ We will identify these power series with the functions that they define from $p\mathbb{N}$ to $\mathbb{Q}_{p}.$ Let $\sigma_p(\emptyset):=1. $ 
 Let $\mathcal{P}_{\sigma}$ denote the module over $\mathcal{P}$ generated by  $\sigma_{p}(\underline{s})$ with $\underline{s}\in \cup_{n} \mathbb{N}^{\times n}.$  Then by the shuffle product formula for series, $\mathcal{P}_{\sigma }$ is an algebra.

%{\color{green} Let $0\leq \alpha <p,$ we will call $f$ a power series function on $\alpha +p\mathbb{N},$ if $f$ is power series function and it is 0 on $\mathbb{N} \setminus (\alpha +p\mathbb{N}).$  For a power series function $f,$ let $f^{\alpha}$ denote the power series function on $\alpha +p\mathbb{N},$ which agrees with $f$ on $\alpha +p\mathbb{N}.$ Let $\mathcal{F}^{\alpha}$ and $\mathcal{F}^{\alpha} _{\sigma}$ be the corresponding algebras.}

 \begin{proposition}
The algebra $\mathcal{P}_{\sigma}$ is free with basis $\{\sigma _{p} (\underline{s}) | \underline{s}\in \cup_{n} \mathbb{N}^{\times n} \}$ as a module over $\mathcal{P}.$ 
  \end{proposition}
 
 \begin{proof}  By induction on $m$ we will show the linear independence of the set $S_{m}:=\{\sigma_{p}(\underline{s}) | d(\underline{s})\leq m  \} .$ For any function $f: p \mathbb{N} \to \mathbb{Q}_{p},$ we let $\delta(f)$ denote the function defined by 
$\delta(f)(n):=f(n+p)-f(n).$    Note that 
 $$
 \delta \sigma_{p}(\alpha_{1},\cdots, \alpha_{m+1})(n)=\frac{1}{n^{\alpha_{m+1}}} \sigma_{p}(\alpha_{1},\cdots, \alpha_{m})(n) .$$
 
 We know the linear independence for  the set $S_{0}=\{ 1\}.$ Assuming that  we know the linear independence for  $S_{m},$ we will prove it for  $S_{m+1}.$ Let us suppose that $\{\sigma_{p}(\alpha_{1},\cdots , \alpha_{m+1})  \} \cup S_{m}$ is linearly dependent over $\mathcal{P}.$ Then we have an expression of the form 
 $$
 \sigma_{p}(\alpha_{1}, \cdots, \alpha_{m+1})=\sum_{d(\alpha) \leq m} a_{\alpha}\sigma_{p}(\alpha),
 $$
 with $a_{\alpha}$ a quotient of elements in $\mathcal{P}.$

 Applying $\delta$ to the last equation we get
 $$
 \frac{1}{z^{\alpha_{m+1}}} \sigma_{p}(\alpha_{1},\cdots, \alpha_{m})=\sum_{d(\alpha)=m} \delta (a_{\alpha}) \sigma_{p}(\alpha)+ \sum _{d(\alpha)<m}b_{\alpha}\sigma_{p}(\alpha).
 $$
 
 The induction hypothesis implies that $\frac{1}{z^{\alpha_{m+1}}} =\delta(a_{\alpha_{1}\cdots \alpha_{m}}),$  and this contradicts the lemma above. 
 
 Next we do an induction on the number of elements $\sigma_{p}(\alpha)$ with $d(\alpha)=m+1,$ and $a_{\alpha}\neq 0.$ Suppose that we have a non-trivial equation 
 $$
 \sum _{d(\alpha) \leq m+1 } a_{\alpha} \sigma_{p}(\alpha)=0.
 $$
 By the induction assumption on $m,$ there is a $\beta=(\alpha_{1},\cdots, \alpha_{m+1})$ such that $a_{\beta}\neq 0.$ Dividing by this and rearranging we get 
 $$
 \sigma_{p}(\beta)+ \sum_{ d(\alpha)=m+1 \atop {\alpha \neq \beta}} b_{\alpha}\sigma_{p}(\alpha)= \sum _{d(\alpha) \leq m} b_{\alpha}\sigma_{p}(\alpha),
 $$
 where $b_{\alpha}$ are quotients of power series functions. 
 Applying $\delta$ to this equation and using induction on the number of $b_{\alpha}\neq 0$ with $d(\alpha)=m+1$ we obtain $\delta(b_{\alpha})=0$ for all $\alpha$ with $d(\alpha)=m+1, $ hence these $b_{\alpha}$ are constant and equal to, say $c_{\alpha}. $
 
 So the last equation can be rewritten as 
 $$
 \sigma_{p}(\beta)+ \sum_{d(\alpha)=m+1 \atop {\alpha \neq\beta}} c_{\alpha}\sigma_{p}(\alpha)= \sum _{d(\alpha) \leq m } b_{\alpha}\sigma_{p}(\alpha),
 $$
 applying $\delta$ we obtain that 
 $$
 \frac{1}{z^{\alpha_{m+1}}}+\sum _{k\in \mathbb{N} \atop {k \neq \alpha_{m+1}}} c_{(\alpha_{1}, \cdots ,\alpha_{m},k) } \frac{1}{z^{k}}=\delta (b_{(\alpha_{1}, \cdots, \alpha_{m})}).
 $$
 The above lemma again gives a contradiction. 
\end{proof}

Let $\mathcal{F}$ denote the algebra of power series functions and $\iota \in \mathcal{F}$ denote the function that sends $n$ to $n.$ Let $\mathcal{F}(\frac{1}{\iota})$ be  the algebra obtained by inverting $\iota.$ Note that $\iota$ is already invertible on the components  $i+p\mathbb{N}$  with $0<i<p.$   Let $\mathcal{F}_{\sigma}$ be the module over $\mathcal{F}$ generated by  $\sigma_{p}(\underline{s})$ with $\underline{s}\in \cup_{n} \mathbb{N}^{\times n}.$ Then by the shuffle product formula for series, $\mathcal{F}_{\sigma }$ is an algebra. Let $\mathcal{F}_{\sigma} (\frac{1}{\iota})=\mathcal{F}_{\sigma} \otimes _{\mathcal{F}} \mathcal{F}(\frac{1}{\iota}).$ 

\begin{corollary}\label{freecorollary}
The algebra $\mathcal{F}_{\sigma}$ (resp. $\mathcal{F}_{\sigma}(\frac{1}{\iota})$) is free with basis $\{\sigma _{p} (\underline{s}) | \underline{s}\in \cup_{n} \mathbb{N}^{\times n} \}$ as a module over $\mathcal{F}$ (resp. $\mathcal{F}(\frac{1}{\iota})$). 
\end{corollary}

\begin{definition}
Let $\mathfrak{r}: \mathcal{F}_{\sigma}\to \mathcal{F}$ denote the projection with respect to the above basis. We will denote the projection $\mathcal{F}_{\sigma}(\frac{1}{\iota})\to \mathcal{F}(\frac{1}{\iota})$ by the same notation. Similarly, let $\mathfrak{s}: \mathcal{F}(\frac{1}{\iota}) \to \mathcal{F}$ denote the projection that has the effect of deleting the principal part of the Laurent series expansion for the component $p\mathbb{N},$ and is identity on the components $i+p\mathbb{N}$ with $0<i<p.$ 

\end{definition}

Let $\underline{s}:=(s_{1},\cdots,s_{k}),$ and $\underline{t}:=(t_{1},\cdots,t_{l}).$ We write $\underline{t} \leq \underline{s}$ if there  exists an increasing function $j: \{1,\cdots, l \} \to \{1,\cdots , k \}$  such that $t_{i} \leq s_{j(i)},$ for all $i.$

\begin{lemma}
 Let $f$ be a power series function and let $g$ be defined as 
 $$
 g(n)=\sum _{0<a<n} f(a)\sigma_{p}(\underline{s})(a)
 $$
 for some $\underline{s}:=(s_{1},\cdots,s_{k}).$ Then 
 $$
 g=\sum _{\underline{t} \leq \underline{s}} f_{\underline{t}} \sigma_{p}(\underline{t}),
 $$
 for some power series functions $f_{\underline{t}}.$ Similarly, 
 if $h$ is defined as 
 $$
 h(n):=\sum_{0<a<n \atop {p|a}} \frac{f(a)}{a^{s}} \sigma_{p}(\underline{s})(a),
 $$ for some $s\geq 1$ then 
 $$
 h=\sum _{\underline{t} \leq \underline{s}'} f_{\underline{t}} \sigma_{p}(\underline{t}),
 $$
 for some power series functions $f_{\underline{t}},$ 
 where $\underline{s}':=(s_{1},\cdots,s_{k},s).$ 
 \end{lemma}

\begin{proof} 
We will prove this by induction on $d(\underline{s}).$ Suppose that $d(\underline{s})=0$ and hence $\sigma_{p}(\underline{s})=1.$ Then for $g$ the assertion follows from Proposition \ref{basic}. Let $f(z)=\sum _{0 \leq i}b_{i} z^{i},$ for $|z|\leq |p|,$ then 
$$
h(n)=b_{0}\sigma_{p}(s)(n)+\cdots+ b_{p-1}\sigma_{p}(1)(n)+\sum _{0<a<n \atop {p|a}}\overline{f}(a),
$$
 where $\overline{f}(z):=b_{p}+b_{p+1}z+\cdots .$ Again the statement follows from Proposition \ref{basic}.
 
Now assume the statement for all $\underline{s}$ with $d(\underline{s})\leq k$ and fix $\underline{s}:=(s_{1},\cdots,s_{k+1}).$  Let $F$ be as in Proposition \ref{basic},  then 
$$
g(n)=F(n-1)\sigma_{p}(\underline{s})(n)-\sum_{0<n_{k+1}<n \atop {p|n_{k+1}}}\frac{F(n_{k+1})}{n_{k+1}^{s_{k+1}}}\sigma_{p}(s_{1},\cdots,s_{k})(n_{k+1})
$$ 
 and the statement follows from the induction hypothesis on $h.$
 
On the other hand, 
\begin{eqnarray*}
h(n)&=&\sum_{0<a<n \atop {p|a}}(\frac{b_{0}}{a^s}+\cdots +\frac{b_{s-1}}{a} +\overline{f}(a))\sigma_{p}(\underline{s})(a) \\
&=& b_{0}\sigma_{p}(\underline{s},s)(n)+\cdots + b_{s-1}\sigma_{p}(\underline{s},1)(n)+\sum_{0<a<n \atop {p|a}}\overline{f}(a)\sigma_{p}(\underline{s})(a)
\end{eqnarray*}
 and the statement follows by the statement that we just proved on $g.$
 \end{proof}

\begin{proposition}
For any $\underline{s}$ and $\underline{m},$ $\sigma(\underline{s};\underline{m}) \in \mathcal{F}_{\sigma}.
$
\end{proposition}

\begin{proof}
We will prove this by induction on $d(\underline{s}).$ If $d(\underline{s})=1,$ then $\sigma(\underline{s},\underline{m})=\sigma_{p}(\underline{s})$ if $m_{1}=0;$ and  $\sigma(\underline{s},\underline{m})\in \mathcal{F}$ otherwise by Proposition \ref{basic}. Suppose we know the result for $d(\underline{s})\leq k,$ and fix $\underline{s}$ with $d(\underline{s})=k+1.$  

Let $\underline{s}=(s_{1},\cdots,s_{k+1}),$ $\underline{s}'=(s_{1},\cdots,s_{k}),$ $\underline{m}=(m_{1},\cdots,m_{k+1}),$ and $\underline{m}'=(m_{1},\cdots,m_{k}).$  Since 
$$\sigma(\underline{s};\underline{m})(n)=\sum _{0<a<n \atop {p| (a-m_{k+1})}} \frac{\sigma(\underline{s}';\underline{m}')(a)}{a^{s_{k+1}}},$$
using the induction hypothesis we realize that we only need to show that functions of the   form 
$$
\sum _{0<a<n \atop {p| (a-m)}}\frac{f(a)}{a^{s}}\sigma_{p}(\underline{t})(a),
$$
with $f$ a power series function, are in $\mathcal{F}_{\sigma}$ and this is exactly the statement of the previous lemma. 
 \end{proof}

In fact, from the proof above it follows that $\sigma(\underline{s};\underline{m})$  is an $\mathcal{F}$-linear combination of $\sigma_{p} (\underline{t})$ with $\underline{t} \leq \underline{s}.$  
 
 \begin{definition}
 For a function $f \in \mathcal{F}_{\sigma},$ let $\tilde{f}:=\mathfrak{r}(f) \in \mathcal{F}.$  We call $\tilde{f}$ the {\it regularization}  of $f.$ Since by the previous proposition $\sigma(\underline{s};\underline{m}) \in \mathcal{F}_{\sigma},$ we let $\tilde{\sigma}(\underline{s};\underline{m})\in \mathcal{F}$ its regularization and  $\underline{\sigma}(\underline{s};\underline{m})=\lim_{n \to 0}\tilde{\sigma}(\underline{s};\underline{m})(n).$ 

 \end{definition}
 
 For a function $f: \mathbb{N} \to \mathbb{Q}_{p}$ and $0 \leq m <p,$ let $f_{[m]}$ denote the function which is equal to $f$ for values $n$ which are congruent to $m$ modulo $p$ and is 0 otherwise.  Recall that $\gamma(\underline{s};\underline{m})(n):=n^{-s_{k}}\cdot\sigma(\underline{s}';\underline{m}')_{[m_{k}]}(n).$  We will define the regularized version $\tilde{\gamma}(\underline{s};\underline{m})$ of $\gamma(\underline{s};\underline{m})$ as follows. If $m_{k}\neq 0,$ then it is defined by $\tilde{\gamma}(\underline{s};\underline{m})(n)=n^{-s_{k}}\cdot\tilde{\sigma}(\underline{s}';\underline{m}')_{[m_{k}]}(n).$ If $m_{k}=0,$ and $p(z)=a_{0}+a_{1}z+\cdots$ is such that $\tilde{\sigma}(\underline{s}';\underline{m}')(n)=p(n)$ for $p|n,$ then 
 $\tilde{\gamma}(\underline{s};\underline{m})(n)=a_{s_{k}}+a_{s_{k}+1}n+\cdots ,$ if $p|n$ and 0 otherwise. Finally we let $\underline{\gamma}(\underline{t};\underline{m})=\lim_{n \to 0}\tilde{\gamma}(\underline{t};\underline{m})(n).$ 
 
 Another way to describe this is as follows. For any $\underline{s}$ and $\underline{m}, $ $\gamma(\underline{s};\underline{m}) \in \mathcal{F}_{\sigma}(\frac{1}{\iota}),  $  and $\tilde{\gamma}(\underline{s};\underline{m}):=\mathfrak{s} \circ \mathfrak{r}(\gamma(\underline{s};\underline{m})).$
  
\begin{remark}\label{derivativeofsigma}
Note that $\tilde{\sigma}(\underline{s};\underline{m})^{(n)}(0)=-n!\cdot\underline{\sigma}(\underline{s},n;\underline{m},0)=n! \cdot \underline{\gamma}(\underline{s},n;\underline{m},0).$  The first identity follows from the fact that if $P_{k}(z)$ is the polynomial such that $P_{k}(n)$ is the sum of the $k$-th powers of the first $n$ positive integers then $(z+1)|P_{k}(z),$ for $k \geq 1.$    
\end{remark}

\begin{definition}
Let $\pazocal{P}_{w}$ (resp. $\pazocal{S}_{w},$ resp. $\tilde{\pazocal{S}}_{w}$) denote the   $\mathbb{Q}$-space spanned  by  the $\underline{\sigma}(\underline{s};\underline{m})$ (resp. $\gamma(\underline{s};\underline{m}),$ resp. $\tilde{\gamma}(\underline{s};\underline{m})$), with $w(\underline{s})= w,$ and $\pazocal{P}:=\sum_{w} \pazocal{P}_{w}$ (resp. $\pazocal{S}:=\sum_{w} \pazocal{S}_{w},$ resp. $\tilde{\pazocal{S}}:=\sum_{w} \tilde{\pazocal{S}}_{w}$).
\end{definition}

Let $\omega_{i}:=dlog(z-i),$ for $i=0,1$ and $\omega_{p}:=dlog(z^{p}-1).$

\begin{lemma} Let $f(z)=\sum _{1 \leq n}a_{n} z^{n},$ such that 
$df=\omega\sum _{1 \leq n} \gamma(\underline{s};\underline{m})(n)z^{n}.$ 

If $\omega=\omega_{0}$ then $a_{n}=\gamma((s_{1},\cdots,s_{k}+1);\underline{m})(n).$  

If $\omega=\omega_{1}$ then $a_{n}=-\sum_{0\leq i \leq p-1}\gamma(\underline{s},1;\underline{m},i)(n).$

If $\omega=\omega_{p}$ then $a_{n}=-p\gamma(\underline{s},1;\underline{m},m_{k})(n).$
\end{lemma}

\begin{proof}
Elementary computation.
\end{proof}

\begin{corollary}\label{corinddif}
Suppose that  $f(z)=\sum _{1 \leq n}q(n) z^{n},$ such that $$df=\omega_{i} \sum _{1\leq i}\alpha(n)z^{n},$$ with $i=0,1$ or $p$ and $\alpha \in \pazocal{S}_{w}.$  Then $q \in \pazocal{S}_{w+1}.$ 
\end{corollary}

\begin{remark}\label{remdiff}
If $g(z)=\sum _{1 \leq n} k(n)z^{n},$ with $k \in \mathcal{F}_{\sigma}(\frac{1}{\iota}),$ then we let $\mathfrak{r}(g)(z):=\sum _{1 \leq n} \mathfrak{r}(k)(n)z^{n}.$ Similarly if $k \in \mathcal{F}(\frac{1}{\iota}),$ then we let $\mathfrak{s}(g)(z):=\sum _{1 \leq n} \mathfrak{s}(k)(n)z^{n}.$ Clearly, we have $\mathfrak{r}(g')=\mathfrak{r}(g)'.$ On the  other hand, in general $\mathfrak{s}(g')\neq \mathfrak{s}(g)'. $ For example for $g(z)=\sum_{1 \leq n} \frac{z^n}{n}, $ the left hand side is $\sum_{0 \leq n} z^n,$ on the other hand the right hand side is   $\sum _{0 \leq n \atop{ p \nmid (n+1)}}z^{n}.$ However, if $\mathfrak{s}(g')=f'$ for some $f=\sum_{1 \leq n}t(n)z^{n},$ with $t \in \mathcal{F},$ then $\mathfrak{s}(g')= \mathfrak{s}(g)'. $ This follows  from the basic observation that if $k(z)$ is a Laurent series such that $\mathfrak{s}(zk(z))=zt(z)$ for some power series $t(z),$ then the coefficient of $1/z$ in  $k(z)$ is 0 and hence $\mathfrak{s}(zk(z))=z\mathfrak{s}(k(z)).$
\end{remark}

\begin{proposition}\label{stildeprop}
Suppose that  $f(z)=\sum _{1 \leq n}q(n) z^{n},$ such that 
\begin{eqnarray}\label{indeqn}
df=\omega_{0} \sum_{1 \leq n} \alpha(n)z^{n}+\omega_{1} \sum_{1 \leq n} \beta(n)z^{n}+\omega_{p} \sum_{1\leq n} \gamma(n)z^{n},
\end{eqnarray} with $\alpha,\beta,\gamma \in \sum_{a+b=w} \pazocal{P}_{a}\cdot \tilde{\pazocal{S}}_{b}.$  If $\lim_{n \to 0}q(n)$ exists then $q \in \sum_{a+b=w+1} \pazocal{P}_{a}\cdot \tilde{\pazocal{S}}_{b}.$ 
\end{proposition}

\begin{proof}
We have
$$
nq(n)=\alpha(n)-\sum_{1 \leq k<n}\beta(k)-p \sum_{1\leq k <n \atop {p|(n-k)}}\gamma(k).
$$
By Proposition \ref{basic}, the function $nq(n)$ is a power series function and hence so is the function $q(n)$ when restricted to $\mathbb{N}\setminus p\mathbb{N}.$ Let $r(z):=\sum _{0\leq i}b_i z^{i}$ be the power series such that $nq(n)=r(n)$ for all $p|n.$ The assumption on the limit implies that $\lim_{n \to 0} \frac{b_{0}}{n}$ exists. Hence $b_{0}=0$ and $q \in \mathcal{F}.$

By Corollary \ref{corinddif}, there is $F(z)=\sum _{1 \leq n} s(n)z^{n}$ such that $s \in \sum _{a+b=w+1}\pazocal{P}_{a} \cdot \pazocal{S}_{b}$ and $\mathfrak{s}\circ \mathfrak{r}(F')dz$ is equal to the right hand side of (\ref{indeqn}), and hence $\mathfrak{s}\circ \mathfrak{r}(F')=f'.$   Let $g=\mathfrak{r}(F).$ Then  
$\mathfrak{s}(g')=\mathfrak{s}(\mathfrak{r}(F)')=\mathfrak{s}(\mathfrak{r}(F'))=f'.$ Since  we proved above that $q \in \mathcal{F},$ we  conclude by Remark \ref{remdiff} that $\mathfrak{s}(g')=\mathfrak{s}(g)'.$   Therefore $\mathfrak{s}(g)'=f',$ and $f=\mathfrak{s}(g)=\mathfrak{s}\circ \mathfrak{r}(F).$ Hence $q$ is of the form as stated above. 
\end{proof}

\begin{proposition}\label{propalg} 
$\pazocal{P}$ is a $\mathbb{Q}$-algebra. 
\end{proposition}

\begin{proof}
The statement follows from the shuffle product formula since this implies that $\pazocal{P}_{a} \cdot \pazocal{P}_{b} \subseteq \pazocal{P}_{a+b}.$
\end{proof}

\section{p-adic multi-zeta values}

In this section we follow the notation of \cite{U1}, except that we denote $g(z)$ by $\mathfrak{g}(z).$  Then letting $h:=\pazocal{F}_{*}(_{t_{\infty 0}}e_{t_{01}}),$ we have $h=\mathfrak{g}(\infty).$  The fundamental equation \cite[(2) p.135]{U1} that connects $g$ and $h$ takes the form:
\begin{eqnarray}\label{fundid}
(e_{0}+e_{1})h=h(e_{0}+g^{-1}e_{1}g)
\end{eqnarray} 
and the fundamental differential equation \cite[(1) p.133]{U1} takes the form:

\begin{eqnarray}\label{funddiff}
d \mathfrak{g}=p(e_0 \mathfrak{g} -\mathfrak{g}e_{0})\cdot\omega_{0}+e_{1} \mathfrak{g} \cdot \pazocal{F}^{*}\omega_{1}-p\mathfrak{g} (g^{-1}e_{1}g)\cdot \omega_{1}. 
\end{eqnarray}

For every $e^{I},$ let $\mathfrak{g}[e^{I}]$ denote the coefficient of $e^I$ in $\mathfrak{g}$ and $\mathfrak{g}\{e^{I}\}$ denote the function that sends $n$ to the coefficient of $z^{n}$ in $\mathfrak{g}[e^{I}].$ If $I=e_{0} ^{i_{1}}e_{1} ^{j_{1}} \cdots e_{0} ^{i_{k}}e_{1} ^{j_{k}},$ let $d(I):=|\{j_{t} | j_{t} \neq 0, \; 1 \leq t \leq k\}|$ and $w(I):=\sum (i_{t}+j_{t}).$  

\begin{theorem}\label{mainthm}
 For each $I,$ $\mathfrak{g}\{e^{I}\} \in \sum _{a+b=w(I)} \pazocal{P}_{a} \cdot \tilde{\pazocal{S}}_{b}$ and  $g[e^I]\in \pazocal{P}_{w(I)}.$    
\end{theorem}

  \begin{proof}

We will prove the statement by induction on $d(I).$   

\begin{lemma}
The statement above is true for $\mathfrak{g}\{ e^{I}\},$ with $d(I) \leq 1;$ for  $\mathfrak{g}\{ e_{1}e_{0} ^{s}e_{1}\},$ with any $s \geq 0;$ and for $g[e^I],$ with $d(I) \leq 1.$
\end{lemma}

\begin{proof}
We proved on  \cite[p. 138]{U1} that $\mathfrak{g}\{ e_{0} ^{s-1} e_{1} \}(n)= \frac{p^s}{n^{s}},$ if $p \nmid n$ and 0 otherwise. Hence $\mathfrak{g}\{ e_{0} ^{s-1} e_{1} \}=p^s\sum _{1\leq i <p}\gamma(s;i) \in \tilde{S}_{s},$ since $\gamma(s;i)=\tilde{\gamma}(s;i), $ for $i \neq 0.$   
The statement for $\mathfrak{g}\{ e^{I}\},$ with $d(I) \leq 1,$ then follows from the fact that $\mathfrak{g}$ is group-like. 

Similarly, we proved on \cite[p. 139]{U1} that $\mathfrak{g}\{e_{1}e_{0}^{s-1}e_{1}\}=$
$$
p^{s+1}((-1)^{s+1}(\sum_{0 \leq i, j <p \atop {i \neq 0}} \gamma(s,1;i,j) ) - \sum _{0<i<p \atop {i \neq 0}} \gamma(s,1;i,i)). 
$$
Clearly, when $i,j \neq 0,$ $\gamma(a,b;i,j) =\tilde{\gamma}(a,b;i,j).$ Note that  for $i\neq 0,$ $\gamma(s,1;i,0)(n)=n^{-1}\sigma(s;i)(n) ,$ if $p|n,$ and is 0 otherwise. Note that $\sigma(s;i)$ is a power series function such that $\lim_{n\to 0}\sigma(s;i)(n)=0$ \cite[p. 139]{U1}.  This implies that,  for $p|n,$ $\sigma(s;i)(n)=\sum _{1 \leq i} a_{i}n^{i},$ for some $a_{i} \in \mathbb{Q}_{p}.$ Therefore $\gamma(s,1;i,0)=\tilde{\gamma}(s,1;i,0).$ Combining these, we deduce that $\mathfrak{g}\{e_{1}e_{0}^{s-1}e_{1}\} \in \tilde{\pazocal{S}}_{s+1}.$ 

Finally, on \cite[p. 140]{U1}, we proved that $g[e_{1}]=0$ and  
$$
g[e_{0}^{s-1}e_{1}]=\frac{p^s}{s-1}\sum_{0<i<p}\sigma(s-1;i)^{(1)}(0),
$$ for $s\geq 1.$ Since $\sigma(s-1;i)=\tilde{\sigma}(s-1;i),$ the claim follows from  Remark \ref{derivativeofsigma}. To deduce the statement for all $I$ with $d(I) \leq 1,$ we use the fact that $g$ is group-like \cite{U1}.
 \end{proof}

 We will prove the result in several steps. Assume  that we know the statement for:

(i) $\mathfrak{g}\{ e^{I}\},$ with $d(I) \leq k;$ \hfill 

(ii) $\mathfrak{g}\{e_{1}e_{0} ^{s_{k}}e_{1} \cdots e_{0} ^{s_{1}}e_1 \},$ for all $0 \leq s_{i};$ and \hfill

(iii) $g[e^I],$ with 
 $d(I) \leq k.$

\begin{lemma}\label{glemmaind}
For $d(I) \leq k+1, $ $\mathfrak{g}\{ e^{I}\} \in  \sum _{a+b=w(I)} \pazocal{P}_{a} \cdot \tilde{\pazocal{S}}_{b}.$
\end{lemma}

\begin{proof}
By (ii), we know the statement for $I=e_{1} e_{0} ^{s_{k}}e_{1} \cdots e_{0} ^{s_{1}}e_{1}.$ It suffices to prove the statement for $e_{0} ^{s_{k+1}}e_{1} e_{0} ^{s_{k}}e_{1} \cdots e_{0} ^{s_{1}}e_{1},$ and we will do this by induction on $s_{k+1}.$ Assume that  $\mathfrak{g}\{e_{0} ^{m}e_{1} e_{0} ^{s_{k}}e_{1} \cdots e_{0} ^{s_{1}}e_{1} \} \in  \sum _{a+b=w} \pazocal{P}_{a} \cdot \tilde{\pazocal{S}}_{b}.$ Comparing the coefficient of $e_{0} ^{m+1}e_{1} e_{0} ^{s_{k}}e_{1} \cdots e_{0} ^{s_{1}}e_{1} $ on both sides of (\ref{funddiff}), and using the inductive hypothesis, together with (i), (iii) and Proposition \ref{propalg}, we see that $d\mathfrak{g}[e_{0} ^{m+1}e_{1} e_{0} ^{s_{k}}e_{1} \cdots e_{0} ^{s_{1}}e_{1}]$ is of the form as in the statement of Proposition \ref{stildeprop}. Furthermore, we note that 
$\lim_{n \to 0}\mathfrak{g}\{e_{0} ^{m+1}e_{1} e_{0} ^{s_{k}}e_{1} \cdots e_{0} ^{s_{1}}e_{1}\}(n) $ exists \cite[Proposition 2]{U1}. Therefore we can apply Proposition \ref{stildeprop} to finish the proof. 
\end{proof}

\begin{proposition}\label{proph}
 If  $d(I) \leq k+1$ then    $h[ e^{I}] \in \pazocal{P}_{w(I)}.$ 
\end{proposition}

\begin{proof} 
Note that $h[e^I]=\mathfrak{g}[e^{I}](\infty)=\lim_{n \to 0}\mathfrak{g}\{ e^{I}\}(n).$ The statement then follows from Lemma \ref{glemmaind}, Remark \ref{derivativeofsigma}, and Proposition \ref{propalg}. 
\end{proof}
The following simple lemma is crucial in what follows. 

\begin{lemma}\label{keylemma}
For $s_{i} \geq 0,$
$g^{-1}[e_{1}e_{0} ^{s_{k+1}}  \cdots e_1e_{0} ^{s_1}]+ g[e_{0} ^{s_{k+1}} e_{1}\cdots e_{0} ^{s_1}e_1] \in \pazocal{P}_{w},$ where $w=w(e_{1}e_{0} ^{s_{k+1}}  \cdots e_1e_{0} ^{s_1})$ . 
\end{lemma}

\begin{proof}
Let us compare the coefficient of $e_{1}e_{0} ^{s_{k+1}} \cdots e_1e_{0} ^{s_1}e_{1}$ on both sides of (\ref{fundid}). 

The left hand side is $h[e_{0} ^{s_{k}+1} \cdots e_1e_{0} ^{s_1}e_{1}],$ which is in $\pazocal{P}_{w}$  by the Proposition \ref{proph}. The right hand side is $(g^{-1}e_{1}g)[e_{1}e_{0} ^{s_{k+1}} \cdots e_1e_{0} ^{s_1}e_{1}], $ which is a sum of the expression in the statement of the lemma and sums of products of the form $g^{-1}[e^{I_{1}}]g[e^{I_{2}}]$ with $d(I_{j}) \leq k,$ and $w(I_{1})+w(I_{2})=w.$   The statement then follows from (iii) and Propostition \ref{propalg}.
\end{proof}

\begin{lemma}\label{gh}
 For $s_{i} \geq 0,$
$h[e_{0} ^{s_{k+1}} e_{1}e_{0} ^{s_{k}}e_{1} \cdots e_{0} ^{s_1}e_{1} ^{2}]+g[e_{0} ^{s_{k+1}+1} e_{1}\cdots e_{0} ^{s_1}e_{1} ] \in \pazocal{P}_{w+1},$ with $w$ as above.
\end{lemma}

\begin{proof}
Let us look at the coefficient of $e_{0} ^{s_{k+1}+1} e_{1}e_{0} ^{s_{k}}e_{1} \cdots e_{0} ^{s_1}e_{1} ^{2}$ with $s_{i} \geq 0$ in (\ref{fundid}). The induction hypothesis, the  fact that $h[e_0]=0,$ and Proposition \ref{proph},   imply that 
$h[e_{0} ^{s_{k+1}} e_{1}e_{0} ^{s_{k}}e_{1} \cdots e_{0} ^{s_1}e_{1} ^{2}]=g^{-1}[e_{0} ^{s_{k+1}+1} e_{1}e_{0} ^{s_{k}}e_{1} \cdots e_{0} ^{s_1}e_{1} ]+({\rm terms \; in \; } \pazocal{P}_{w+1}).$ Then using Lemma \ref{keylemma} we have the statement. 
\end{proof}

\begin{lemma}\label{h1}
For every $s_{i} \geq 0$ there exist $t(\alpha_{k+1},\cdots, \alpha_{1} ) \in \mathbb{Q},$ such that 
$$
h[e_{0} ^{s_{k+2}}e_{1} \cdots e_{0} ^{s_{1}}e_{1}
]-\sum t(\alpha_{k+1},\cdots, \alpha_{1} )h[e_{1}e_{0} ^{\alpha_{k+1}} e_1\cdots e_{0} ^{\alpha_{1}}e_{1}]
$$
is in $\pazocal{P}_{w},$ where $w=w(e_{0} ^{s_{k+2}}e_{1} \cdots e_{0} ^{s_{1}}e_{1})$ and the sum is over $\underline{\alpha}=(\alpha_{k+1}, \cdots\alpha_{1})$ with $w(\underline{\alpha})=\sum s_{i}.$
\end{lemma}

\begin{proof}
There is nothing to prove if $s_{k+2}=0,$ so we assume that $s_{k+2}>0.$ Looking at the coefficient of $e_{0} ^{s_{k+2}}e_{1} \cdots e_{0} ^{s_{1}}e_{1}e_{0}$ on both sides of (\ref{fundid}), and using the induction hypotheses together with the fact that $h[e_{0}]=g[e_{0}]=0,$ we find that 
$$
h[e_{0} ^{s_{k+2}-1}e_{1} \cdots e_{0} ^{s_{1}}e_{1}e_{0}]-h[e_{0} ^{s_{k+2}}e_{1} \cdots e_{0} ^{s_{1}}e_{1}] \in \pazocal{P}_{w}.
$$
 Next using the shuffle formula for $0=h[e_{0} ^{s_{k+2}-1}e_{1} \cdots e_{0} ^{s_{1}}e_{1}]h[e_{0}],$ and the above fact, we obtain that 
 $$
(s_{k+2}+1) h[e_{0} ^{s_{k+2}}e_{1} \cdots e_{0} ^{s_{1}}e_{1}]+\sum _{1\leq i\leq k+1}(s_{i}+1)h[e_{0} ^{s_{k+2}-1}e_{1} \cdots e_0 ^{s_{i}+1}e_1 \cdots e_{0} ^{s_{1}}e_{1}]
 $$
 is in $\pazocal{P}_{w}.$ From this the assertion follows by induction on $s_{k+2}.$ 
\end{proof}

 \begin{lemma} For $s_{i} \geq 0,$
$\mathfrak{g}\{e_{1}e_{0} ^{s_{k+1}} e_{1}e_{0} ^{s_{k}}e_{1} \cdots e_{0} ^{s_1}e_{1} \} \in \tilde{\pazocal{S}}_{w},$ where we have  $w=w(e_{1}e_{0} ^{s_{k+1}} e_{1}e_{0} ^{s_{k}}e_{1} \cdots e_{0} ^{s_1}e_{1}).$  
\end{lemma}

\begin{proof} 
Let us look at the coefficient of $e_{1}e_{0} ^{s_{k+1}} e_{1}e_{0} ^{s_{k}}e_{1} \cdots e_{0} ^{s_1}e_{1}$ in the differential equation (\ref{funddiff}) which gives 
$d \mathfrak{g}[e_{1}e_{0} ^{s_{k+1}} e_{1} \cdots e_{0} ^{s_1}e_{1}]=$
$$
\mathfrak{g}[e_{0} ^{s_{k+1}} e_{1} \cdots e_{0} ^{s_1}e_{1}]\cdot \pazocal{F}^{*}\omega_{1}-p (\mathfrak{g}(g^{-1}e_{1}g))[e_{1}e_{0} ^{s_{k+1}} e_{1}e_{0} ^{s_{k}}e_{1} \cdots e_{0} ^{s_1}e_{1}]\cdot \omega_{1}.
$$
By the induction all the terms contribute to give a term in the form that we are seeking except possibly the term $-p\omega_{1} (g^{-1}[e_{1}e_{0} ^{s_{k+1}}  \cdots e_1e_{0} ^{s_1}]+ g[e_{0} ^{s_{k+1}} e_{1}\cdots e_{0} ^{s_1}e_1]).$ But this is also in the form that we were looking for by Lemma \ref{keylemma}.
\end{proof}

    \begin{corollary}\label{h2}
     For $s_{i} \geq 0,$
$h[e_{1}e_{0} ^{s_{k+1}} e_{1} \cdots e_{0} ^{s_1}e_{1} ] \in \pazocal{P}_{w},$ where we have $w=w(e_{1}e_{0} ^{s_{k+1}} e_{1} \cdots e_{0} ^{s_1}e_{1}).$   
\end{corollary}
\begin{proof}
Clear using the fact that $h=\mathfrak{g}(\infty).$ 
\end{proof}

\begin{lemma}
For $s_{i} \geq 0,$ $g[e_{0} ^{s_{k+1}+1} e_{1} \cdots e_{0} ^{s_1}e_{1} ] \in \pazocal{P}_{w},$ where we have  $w=w(e_{0} ^{s_{k+1}+1} e_{1} \cdots e_{0} ^{s_1}e_{1}).$
 \end{lemma}
 \begin{proof}
 This follows from combining Lemma \ref{gh}, Lemma \ref{h1}, and Corollary \ref{h2}.
 \end{proof}

\begin{lemma}
We have $g[ e_{1}e_{0} ^{s_k}e_{1}  \cdots e_{0} ^{s_1}e_{1} ] \in \pazocal{P}_{w},$ with $w=w(e_{1}e_{0} ^{s_{k}}  \cdots e_{0} ^{s_1}e_{1}).$
 \end{lemma}
 \begin{proof}
If all $s_{i}=0,$ then the expression is 0. Otherwise let $j:=max\{i|s_{i} \neq 0 \}.$ Then applying Lemma \ref{keylemma} several times we see that it is sufficient to prove that $g[e_0 ^{s_{j}}e_{1}\cdots e_{0} ^{s_{1} } e_{1} ^{k+2-j}] \in \pazocal{P}_{w},$ which we did in the previous proposition.  
 \end{proof}
This finishes the proof of the theorem.
\end{proof} 

Recall that the p-adic multi-zeta values $\zeta_{p}(s_{k},\cdots ,s_{1})$  were defined as
\begin{eqnarray}\label{pzeta}
g[e_{0} ^{s_{k}-1}e_{1}\cdots e_{0} ^{s_{1}-1}e_{1}]=p^{\sum s_{i}}\zeta_{p}(s_{k},\cdots ,s_{1})
\end{eqnarray}
in \cite[Definition 3]{U1}.

Let $\pazocal{Z}$ denote the $\mathbb{Q}$-space generated by the p-adic multi-zeta values. By the shuffle product formula this is an algebra.

\begin{theorem}
We have the inclusion $\pazocal{Z} \subseteq \pazocal{P}.$ 
\end{theorem}

\begin{proof}
This is a consequence of Theorem \ref{mainthm} and (\ref{pzeta}).
\end{proof}

\begin{remark}
 Furusho defined p-adic multiple-zeta values  using  Coleman's theory of iterated p-adic  integrals. Our approach in \cite{U1} and here is based on Deligne's theory of the comparison isomorphism between  the de Rham and the crystalline fundamental group.   However, the $\mathbb{Q}$-space that these two different definitions generate are the same \cite[Theorem 2.8, Examples 2.10]{fu} and hence the p-adic multiple-zeta values as defined by Furusho also lie in $\pazocal{P}.$  \end{remark}

 \end{document}